\documentclass[11pt]{amsart}

\usepackage{amssymb,amsmath,latexsym}
\usepackage{amsthm}
\usepackage{amsfonts}
\usepackage{enumerate}
\usepackage{booktabs}
\usepackage{mathrsfs}

\usepackage{hyperref}
\usepackage{graphicx,pgf}
\usepackage{float}

\usepackage{caption}
\usepackage{subcaption}

\usepackage{fancyhdr}

\theoremstyle{plain} 
\newtheorem{theorem}{Theorem}[section]
\newtheorem{proposition}[theorem]{Proposition}
\newtheorem{lemma}[theorem]{Lemma}

\newtheorem{remark}{Remark}

\theoremstyle{definition}

\theoremstyle{remark}

\numberwithin{equation}{section}

\begin{document}
\title{Continuity and  growth of free multiplicative convolution semigroups}
\author{Xiaoxue Deng}
\address{Xiaoxue Deng: School of Mathematics and Statistics, Wuhan University, No. 299 BaYi Road, Wuhan, Hubei 430074, China}
\email{dengxx@whu.edu.cn}
\author{Ping Zhong}
\address{Ping Zhong: School of Mathematics and Statistics, Wuhan University, No. 299 BaYi Road, Wuhan, Hubei 430074, China; and Department of Pure Mathematics, University of Waterloo, Waterloo, ON, Canada}
\email{ping.zhong@uwaterloo.ca}

\begin{abstract}
Let $\mu$ be a compactly supported probability measure on the positive half-line and let $\mu^{\boxtimes t}$ be
the free multiplicative convolution semigroup. We show that the support of $\mu^{\boxtimes t}$
varies continuously as $t$ changes. We also obtain the asymptotic length of the support of 
these measures. 
\end{abstract}

\maketitle

\section{Introduction}
Let $\mu$ and $\nu$ be probability measures on $[0,\infty)$. 
The free convolution $\mu\boxtimes \nu$ represents 
the distribution of product of two positive operators
in a tracial $W^*$-probability space
whose distributions are $\mu$ and $\nu$ respectively. 
We refer to \cite{Basic} for an introduction to free probability theory. 

Given a probability measure $\mu$ on $[0,\infty)$ not being a Dirac measure, it is known \cite{BB2005}  that, for any $t>1$, the 
fractional free convolution power $\mu^{\boxtimes t}$ is defined appropriately, such that
it interpolates the discrete convolution semigroup $\{ \mu^{\boxtimes n}\}_{n\in \mathbb{Z}^+}$, where
$\mu^{\boxtimes n}={\mu\boxtimes \cdots\boxtimes\mu}$ is the $n$-fold free convolution.
This is the multiplicative analogue of Nica-Speicher semigroup \cite{NicaS1996} defined firstly for the free additive convolution. 
The free convolution semigroups obey many regularity properties and has been studied extensively. See \cite{Lectures16}
for a survey on free convolutions and other topics in free probability theory. 

Denote by $\mathrm{supp}(\mu)$ the support of the measure $\mu$. 
We will prove the following result about continuity of the support of $\mu^{\boxtimes t}$,
which is the analogue of the work \cite{John2016} for free multiplicative convolution on the positive half line. 
\begin{theorem}\label{thm:0.1}
	Let $\mu$ be a compactly supported probability measure on $[0, \infty)$. Then the 
	supports of measures $\{\mathrm{supp}(\mu^{\boxtimes t} \}$ change continuously in the Hausdorff metric
	with respect to the parameter $t$. 
\end{theorem}

We then study the asymptotic size of the support. 
\begin{theorem}\label{thm:0.2}
	Let $\mu$ be a compactly supported probability measure on $[0, \infty)$ with mean 
	$m_1(\mu)=\int_0^\infty s \,d\mu(s)=1$.
	We denote by 
	$
	|| \mu^{\boxtimes t}|| =\max\{m: m\in\mathrm{supp}(\mu^{\boxtimes t}) \}.
	$
	Then 
	\[
	\lim_{t\rightarrow \infty}{||\mu^{\boxtimes t} ||}/{t}=eV.
	\]
	where $V$ is the variance of $\mu$. 
\end{theorem}
Theorem \ref{thm:0.2} generalizes Kargin's work \cite{Kargin2008} (see also \cite{OctavioC2012}) to continuous semigroup. Our proof is different from Kargin's proof, but uses the density formula for free convolution semigroups \cite{HZ2014, zhong6}. 
The free multiplicative convolution on the unit circle is usually studied together with the positive half line case. It was shown in \cite{exp2016} that many results can be deduced from results on free additive convolutions. Since analogue results were known in additive case, a separate work on free multiplicative convolution on the unit circle become unnecessary. Hence we focus on measures on the positive half line in our article.

An estimation for the size of support of free additive convolution semigroups was obtained in \cite{BCN2012, zhong5}. In light of the proof of Theorem \ref{thm:0.2}, it is very likely higher order asymptotic expansion can be obtained for free additive convolution semigroups. We plan to investigate it in a forthcoming work as it may have some applications in quantum information theory.

The paper is organized as follows. In Section 2, we collect some known regularity results about free multiplicative convolution semigroups. In Section 3, we give the proof for Theorem \ref{thm:0.1}. We study the asymptotic behavior of free multiplicative convolution semigroups
and give the proof for Theorem \ref{thm:0.2} in Section 4.

\section{Free convolution on the positive half line }
 
Let $\mu$ be a probability measure
on $[0,\infty)$.
The $\psi$-transform of 
$\mu$ is the moment generating function of $\mu$ defined as
\[\psi_\mu(z)=\int_0^\infty\frac{zs}{1-zs}\;d\mu(s),\] which
is analytic on $\Omega=\mathbb{C}\backslash[0,\infty)$. The $\eta$-transform of $\mu$ is
defined as $\eta_\mu=\psi_\mu/(1+\psi_\mu)$
on the same domain as the
$\psi$-transform. 

Any probability measure $\mu$ on $[0,\infty)$ can be recovered from
its $\eta$-transform by Stieltjes inversion formula. Indeed, we have the identity
\begin{equation}\label{eq:identity}
G_\mu\left(\frac{1}{z}\right)=\frac{z}{1-\eta_\mu(z)},\;\;\;\;\;z\in\Omega,
\end{equation}
where $G_\mu$ is the Cauchy transform of $\mu$.
If $\mu$ is not a Dirac measure, then $\eta_\mu'(z)>0$
for $z<0$, and therefore $\eta_\mu|(-\infty,0)$ is invertible. Let
$\eta_\mu^{-1}$ be the inverse of $\eta_\mu$ and set
$\Sigma_\mu(z)={\eta_\mu^{-1}(z)}/{z}$,
where $z<0$ is sufficiently small. 
The free convolution of two such probability measures
$\mu$ and $\nu$ is determined by $\Sigma_{\mu\boxtimes \nu}(z)=\Sigma_\mu(z)\Sigma_\nu(z)$. 
In particular, the $n$-th order free multiplicative convolution power
$\mu^{\boxtimes n}$ of $\mu$ satisfies the identity
\begin{equation}\label{eq:02.1}
\Sigma_{\mu^{\boxtimes n}}(z)=\Sigma_\mu^n(z),
\end{equation}
 where $z<0$ is sufficiently small. 

We now briefly recall the construction of $\mu^{\boxtimes t}$ which interpolates the relation (\ref{eq:02.1}) as follows. 
Let $\kappa_{\mu}(z)=z/\eta_{\mu}(z)$ for $z\in \Omega$
and one can write $\kappa_\mu(z)=\exp(u(z))$, where
$u$ is an analytic function on $\Omega$ and can be expressed as
\begin{equation}\label{eq:rep-2.1}
u(z)=a+\int_0^\infty \frac{1+zs}{z-s}d\rho(s),
\end{equation}
where $a=-\log|\eta_{\mu}(i)|$ and $\rho$ is a finite positive Borel measure on $[0, \infty)$ following
\cite[Proposition 4.1]{HZ2014}.

We define $
\Phi_t(z):=z\exp[(t-1)u(z)]$. 
It turned out that the function $\Phi_t$ is the right inverse of Voiculescu subordination function $\omega_t$ \cite{BB2005}. 
More precisely, we have
\[\Phi_t(\omega_t(z))=z, \quad \text{and}\quad \eta_{\mu^{\boxtimes t}}(z)=\eta_{\mu}(\omega_t(z)) \]
for all $z\in \Omega$ and $t>1$. 
It turns out that the function $\omega_t$ can be regarded
as the $\eta$-transform of a $\boxtimes$-infinitely divisible measure on $[0,\infty)$ and the function
$\eta_{\mu^{\boxtimes t}}$ can be retrieved from $\omega_t$. We refer to \cite{BB2005, HZ2014} for more details. 

The following result was proved in \cite{BB2005}.
\begin{theorem}\label{thm:2.1}
Let $\mu$ be a probability measure on $[0,\infty)$ and $t>1$. 
\begin{enumerate}
\item A point $x\in (0,\infty)$ satisfies $\eta_{\mu^{\boxtimes t}}(x)=1$ if and only if $x^{-1/t}$ is an atom 
of $\mu$ with mass $\mu ( \{ x^{-1/t}\})\geq (t-1)/t$. If $\mu(\{ x^{-1/t}\})>(t-1)/t$, then $1/x$ 
is an atom of $\mu^{\boxtimes t}$, and 
\[
   \mu^{\boxtimes t}\left(\left\{ 1/x\right\} \right)=t\mu (\{x^{-1/t} \})-(t-1).
\]
\item The nonatomic part of $\mu^{\boxtimes t}$ is absolutely continuous and its density is continuous except at the 
finitely many points $x$ such that $\eta_{\boxtimes t}(x)=1$.
\item The density of $\mu^{\boxtimes t}$ is analytic at all points where it is different from zero. 
\end{enumerate}
\end{theorem}
The study of regularity property of free convolutions relies on Voiculescu's subordination result \cite{BB2005, BB2007new, Biane1998, DVV1993}. 
By a careful study of boundary behavior of subordination functions,
we were able to give a formula for the density function of absolutely continuous part of $\mu^{\boxtimes t}$ in \cite{HZ2014}.
To describe our result, we need some auxiliary functions studied in \cite{HZ2014}. 

Let $g$ be a function defined on $(0,\infty)\times (0,\pi)$ by
\begin{equation}\label{eq:02.3}
 g(r,\theta)=-\frac{\Im u(re^{i\theta})}{\theta}=\frac{r\sin\theta}{\theta}\int_0^\infty\frac{s^2+1}{r^2-2rs\cos\theta+s^2}\;d\rho(s).
\end{equation}
The function $g(r,\theta)$ is decreasing on $(0,\pi)$ for any
$r\in (0,\infty)$ fixed and $\lim_{\theta\to\pi^-}g(r,\theta)=0$.
We then set
\[g(r)=\lim_{\theta\to 0}g(r,\theta)=\int_0^\infty\frac{r(s^2+1)}{(r-s)^2}\;d\rho(s)\]
and
\[A_t(r)=\inf\left\{\theta\in(0,\pi):g(r,\theta)<\frac{1}{t-1}\right\}.\] 
It is clear that if $A_t(r)>0$, then $g(r,A_t(r))=1/(t-1)$. 
We further let $ h_t(r):=\Phi_t(re^{i A_t(r)})$ and 
\[V_t^+=\left\{r\in(0,\infty):g(r)>\frac{1}{t-1}\right\}.\]
The function $h_t$ is a homeomorphism of $[0,\infty)$ and
$\lim_{r\rightarrow \infty}h_t(r)=\infty$. 

The functions defined above can be used to describe the image set $\Omega_t=\omega_t(\mathbb{C}^+)$. The set  $\Omega_t$ is in fact the connected component of $\Phi_t^{-1}(\mathbb{C}^+\cup (-\infty,0))$ 
having the negative half plane as part of its boundary.
Moreover, we proved that
\begin{equation}\nonumber
\Omega_t=\{re^{i\theta}: A_t(r)<\theta<\pi, r\in (0,\infty) \},
\end{equation}
and 
\begin{equation}\nonumber
\partial\Omega_t=(-\infty, 0]\cup\{re^{i\theta}: \theta=A_t(r), r\in (0,\infty) \}.
\end{equation}
The following result is one of main results in \cite{HZ2014}. 
\begin{theorem} \label{density1}
Suppose that $\mu$ is a probability measure on $[0,\infty)$ not being a Dirac measure and $t>1$. Let
$S_t=\left\{1/{h_t(r)}:r\in V_t^+\right\}$.
Then the following statements hold.
\begin{enumerate} [$(1)$]
\item {The measure $(\mu^{\boxtimes t})^{\mathrm{ac}}$ is concentrated on the closure of $S_t$.}
\item {The density of $(\mu^{\boxtimes t})^{\mathrm{ac}}$ is analytic on the set
$S_t$ and is given by
\[\frac{d(\mu^{\boxtimes t})^{\mathrm{ac}}}{dx}\left(\frac{1}{h_t(r)}\right)=
\frac{1}{\pi}\frac{h_t(r)l_t(r)\sin\theta_t(r)}{1-2l_t(r)\cos\theta_t(r)+l_t^2(r)},\;\;\;\;\;r\in
V_t^+,\] where $l_t(r)=r\exp\Re u(re^{iA_t(r)})$
and $\theta_t(r)={tA_t(r)}/{(t-1)}$
for $r\in V_t^+$.}
\item {The number of components in $\mathrm{supp}(\mu^{\boxtimes t})^{\mathrm{ac}}$ is a decreasing function of $t$.}
\end{enumerate}
\end{theorem}

\section{Continuity of free convolution semigroups }
In this section, we assume that the probability measure $\mu$  on $[0,\infty)$
is compactly supported.
\begin{lemma}\label{lemma:3.1}
Let $\mu$ be a probability measure on $[0, \infty)$. Set $\kappa_\mu(z)=z/\eta_\mu(z)$ and
write $\kappa_\mu(z)=\exp[u(z)]$, where $u$ is give by (\ref{eq:rep-2.1}). Then 
$\lim_{x\rightarrow 0^-} \kappa(x)=1/m_1(\mu)$ and
\[
  \int_0^\infty \frac{1}{s}d\rho(s)=\log m_1(\mu)+a. 
\]
Moreover, when $m_1(\mu)=1$, we have
\[\int_0^\infty\frac{1+s^2}{s^2}\,d\rho(s)=V,\]
where $V$ is the variance of $\mu$.
\end{lemma}
\begin{proof}
Observe that $\lim_{r\rightarrow 0^-}\eta_\mu(r)/r=m_1(\mu)$ by the definition of $\eta_\mu$. We write $u$ as
\[
  u(z)=a-z\int_0^\infty1\,d\rho(s)-(z^2+1)\int_0^\infty\frac{1}{s-z}\,d\rho(s),
\]
and hence $\lim_{r\rightarrow 0^-}u(r)=a-\int_0^\infty 1/s \,d\rho(s)$. We then deduce the first equation. 

We calculate
\begin{align*}
u'(r)=-\int_0^\infty \frac{s^2+1}{(r-s)^2}\,d\rho(s),
\end{align*} 
for all $r<0$ and, by Monotone Convergence Theorem,
\[
\lim_{r\rightarrow 0^-}u'(r)=-\int_0^\infty \frac{s^2+1}{s^2}\,d\rho(s).
\]
On the other hand, when $m_1(\mu)=1$, we have
$u(z)=-\ln\left({\eta_\mu(z)}/{z}\right)=-\ln(1+V z+o(z))$,
which yields that
$u'(0^-)=-V$. 
Therefore, we deduce that
\[
 V=\int_0^\infty\frac{s^2+1}{s^2}\,d\rho(s).
\]
This finishes the proof. 
\end{proof}

The following result is from \cite{HZ2014}.
\begin{lemma}\label{lemma:3.2}
Let $I$ be an open interval contained in $(V_t^+)^c$, then $\rho(I)=0$ and $g$ is strictly convex on $I$. 
\end{lemma}

\begin{proposition}\label{prop:3.4}
Let $\mu$ be a compactly supported probability measure on $[0, \infty)$ and $1<\alpha<\beta$. Write $\kappa_\mu(z)=\exp[u(z)]$, where $u$ is an analytic function given by (\ref{eq:rep-2.1}). 
\begin{enumerate}[$(1)$]
\item Then there exists $a>0$ such that $[0,a)\cap\mathrm{supp}(\rho)=\emptyset$. 
Moreover, the set $V_t^+$ is uniformly bounded away from zero and increasing for all $t\in (\alpha, \beta)$.
\item Given any $b>0$, the sets $V_t^+\cap [0,b]$
are Hausdorff continuous with respect to $t$. 
\end{enumerate}
\end{proposition}
\begin{proof}
By the definition of $\eta$-transform, we have $\eta_\mu(0)=0$ and $\eta_\mu'(0)=m_1(\mu)<\infty$. 
The identity 
\[
  \frac{z}{1-\eta_\mu(z)}=G_\mu\left(\frac{1}{z}\right)=
    \int_0^\infty \frac{1}{1/z-x}\,d\mu(x),
\]
implies that $\eta_\mu$ is real on some interval $[0,a)$ under the assumption that $\mu$ is compactly supported. It follows that $u$ is also real on $[0,a)$ and hence $\rho([0,a))=0$
by Stieltjes inversion formula. 
Hence, the function $g$ is finite in a neighborhood of zero and $g(r)=\int_0^\infty\frac{r(s^2+1)}{(r-s)^2}\;d\rho(s)\rightarrow 0$ as $r\rightarrow 0$.
We see that $V_t^+$ is bounded away from $0$. 
The definition of $V_t^+$ immediately implies that $V_{t_1}^+\subset V_{t_2}^+$ if $t_1<t_2$. 
 This proves the first assertion. 

Assume now $t_n \rightarrow t \in (\alpha, \beta)$ and $(V_{t_n}^+\cap [0,b])\not\subset B_\epsilon (V_t^+\cap [0,b])$, where $B_\epsilon (V_t^+\cap [0,b])$
is a $\epsilon$-neighborhood of the set $V_t^+\cap [0,b]$.
We then have a series $r_n \in (V_{t_n}^+\cap [0,b])\backslash B_\epsilon (V_t^+)$.
We may assume that 
$r_n \rightarrow r$ by passing to a subsequence if necessary. Lemma 3.3 implies that
$\mathrm{supp}( \rho) \subset B_\epsilon (V_t^+)$ and hence we can take the limit
\[
g(r)=\lim_{n\rightarrow\infty} g(r_n)= \lim_{n\rightarrow\infty}\frac{1}{t_n-1}.
\]
On the other hand, $g\leq  \frac{1}{t-1}$ and $g$ is strictly convex on any open interval contained in $(V_t^+)^c$.
This contradiction proves the second claim. 
\end{proof}

\begin{proposition}\label{prop:3.4}
Given $b>0$, the graphs $\{ re^{i A_t(r)}: r\in V_t^+\cap (0, b) \}$ are continuous in the Hausdorff metric for $t\in (1, \infty)$. 
\end{proposition}
\begin{proof}
For $0<c<\pi$, we define $V_{t,c}^+=\{r\in V_t^+: A_t(r)\geq c \}$.
Given $\epsilon>0$, we will start to prove that there exists $\delta>0$ such that
\begin{equation}\label{eq:002.1}
 A_{t}(r)<A_{s}(r)\leq (1+\epsilon)A_t(r)
\end{equation}
for all $r\in V_{t,c}^+\cap (0,b)$ if $0<s-t<\delta$. 
The first inequality follows from the fact that the function $g(r,\theta)$ is a decreasing function of $\theta$. To prove the second inequality 
by contradiction, we assume that there exists a series $t_n>t$, $t_n\rightarrow t$ and $r_n, r \in \overline{V_{t,c}^+\cap (0,b)}$ such that $r_n\rightarrow r$ and 
$A_{t_n}(r_n) >(1+\epsilon) A_t(r_n)$. We then have
\[
 g(r_n, (1+\epsilon)A_t(r_n))\geq g(r_n, A_{t_n}(r_n))=\frac{1}{t_n-1}.
\]
As $A_t(r_n)\geq c$, we can take the limit and obtain
\[
g(r, (1+\epsilon)A_t(r))\geq \frac{1}{t-1},
\]
due to the fact that the integrand in (\ref{eq:02.3}) is bounded away from zero and $A_t$ is continuous. 
On the other hand, we have $g(r, A_t(r))=\frac{1}{t-1}$ and $g(r,\theta)$ is a strictly decreasing function of $\theta$. 
This contradiction yields the second inequality in (\ref{eq:002.1}).

Given $\epsilon>0$, using the similar argument as above, we can prove that 
there exists $\delta>0$ such that
\begin{equation}\label{eq:002.2}
 (1-\epsilon)A_t(r) \leq A_{s}(r)<A_{t}(r)
 \end{equation}
for all $r\in V_{t,c}^+\cap (0,b)$ if $-\delta<s-t<0$. 

We claim that 
\[\sup\{A_{s}(r): r\in (0, b)\backslash V_{t,c}^+\}\leq 2c\]
 if $s-t$ is small enough. Assume that is not the case, then there 
 exists a series $t_n \rightarrow t$ and $r_n\rightarrow r$, where
$r_n \notin V_{t,c}^+\cap (0,b)$, such that $A_{t_n}(r_n)>2c$. We have
\[
  g(r_n, 2c) \geq g(r_n, A_{t_n}(r_n))=\frac{1}{t_n-1}.
\]
Taking the limit, we have $g(r,2c)\geq 1/(t-1)$, which implies
that $A_t(r)\geq 2c$. As the cluster set of $r_n\notin V_{t,c}^+$, $r$ is either not in $V_{t,c}^+$
or an end point of $V_{t,c}^+$, it must satisfy $A_t(r)\leq c$. 
This contradiction proves our claim. 

The desired assertion follows by above results and applying Proposition \ref{prop:3.4}.
\end{proof}

\begin{proposition}\label{prop:3.5}
Let $\mu$ be a compactly supported probability measure on $[0, \infty)$. 
If $\mu(\{ 0\})>1$, then $\mu^{\boxtimes t}(\{ 0\})=\mu(\{ 0\})$ for all $t>1$. 
If $0\in \mathrm{supp}((\mu^{\boxtimes s})^{ac})$ for some $s>1$, then
$0\in \mathrm{supp}((\mu^{\boxtimes t})^{ac})$ for all $t>1$. 
\end{proposition}
\begin{proof}
It follows from \cite{MR1639647} that $rG_\mu(r)\rightarrow \mu(\{ 0\})$
as $r\rightarrow 0^-$ on the negative half line.  
Hence, $\mu(\{ 0\})=1+\lim_{r\rightarrow -\infty}\psi_\mu (r)$.
For any $t>0$, as $\lim_{r\rightarrow -\infty}\omega_t(r)=-\infty$,
we have
\[
\mu^{\boxtimes t}(\{ 0\})=1+\lim_{r\rightarrow -\infty}\psi_{\mu} (\omega_t(r))=1+\lim_{r\rightarrow -\infty}\psi_\mu (r)
=\mu(\{0 \}).
\]

Assuming that $0\in \mathrm{supp}((\mu^{\boxtimes s})^{ac})$ for some $s>1$,
we claim that $\rho$ is not compactly supported. Indeed, if $\rho$ is compactly supported, then 
$\lim_{r\rightarrow \infty}g(r)=0$ and hence the set $A_s$ is compact and 
the set $\mathrm{supp}((\mu^{\boxtimes s})^{ac})$ is bounded away from zero by Theorem \ref{density1}, which contradicts to the assumption. 
The fact that the measure $\rho$ is not compactly supported in turn implies that $V_t$ is not compact and 
$0\in \mathrm{supp}((\mu^{\boxtimes t})^{ac})$ for all $t>1$ by according to Theorem \ref{density1}.
\end{proof}

We are now in a position to prove the main result of this section. 
\begin{proof}[Proof of Theorem \ref{thm:0.1}]
We first focus on the absolutely continuous part. 
Given $t_0>1$, $\delta\in (0,t_0-1)$ and choose $a$ such that $V_{t_0+\delta}^+\subset [a,\infty)$.
For $s\in (t_0-\delta, t_0+\delta)$, we will prove that
\begin{equation}\label{eq:3.3}
 \left| \frac{h_{t_0}(x)}{h_s(y)} \right|= \left| \frac{\Phi_{t_0}(xe^{i A_{t_0}(x)})}{\Phi_s(y e^{i A_s(y)})} \right|=O\left( \frac{x}{y} \right),
\end{equation}
for $x\in V_{t_0}^+$ and $y\in V_s^+\cap V_{t_0}^+$. We fix $x_0\in V_{t_0}^+$ and write
\begin{align*}
\frac{\Phi_{t_0}(x_0e^{i A_t(x_0)})}{\Phi_s(y e^{i A_s(y)})}&=\frac{\Phi_{t_0}(ye^{i A_s(y)})}{\Phi_s(y e^{i A_s(y)})}\frac{\Phi_{t_0}(x_0e^{i A_t(x_0)})}{\Phi_{t_0}(ye^{i A_s(y)})}.
\end{align*}
We observe that the function $(s,x)\rightarrow A_s(x)$ is continuous in a neighborhood of $(t_0, x_0)$ such that $x_0\in V_{t_0}^+$
by the definition (\ref{eq:02.3}). When $\delta$ is sufficiently small, we thus have
\[
\left | \frac{\Phi_{t_0}(ye^{i A_s(y)})}{\Phi_s(y e^{i A_s(y)})}  \right |=|\exp[(t_0-s) u(y e^{i A_s(y)})]|\leq \exp[(t_0-s) C],
\]
where $C$ is a constant depending on $\delta$ and $y$. 
To estimate the second factor, denote $z_1=x_0e^{i A_t(x_0)}$ and $z_2=ye^{i A_s(y)}$.
We have
\[
\left|\frac{\Phi_{t_0}(z_1)}{\Phi_{t_0}(z_2)}\right|=\left(\frac{x_0}{y}\right)\cdot \exp[(t_0-1)(u(z_1)-u(z_2))]
\]
and
\[
 u'(z)=-\int_0^\infty\frac{1+s^2}{(z-s)^2}d\sigma(s).
\]
For any $z=re^{i\theta}\in \Omega_{s}\cap \Omega_{t_0}$, we have
that 
\[
\frac{r\sin(\theta)}{\theta}\int_0^\infty \frac{s^2+1}{|z-s|^2}d\sigma(s)<\frac{1}{t_0-\delta-1},
\]
which yields that
\begin{align*}
|u'(z)|&\leq \int_0^\infty\frac{1+s^2}{|z-s|^2}|dz|\\
   &\leq \frac{\theta}{r\sin(\theta)}\cdot \frac{1}{t_0-\delta-1}.
\end{align*}
We now choose a curve $\gamma\subset  \Omega_{s}\cap \Omega_{t_0}$ connecting $z_1$ and $z_2$, and obtain 
\begin{align*}
| u(z_1)-u(z_2)|&=\left|\int_{\gamma} u'(z)dz\right|\leq \frac{1}{t_0-\delta-1}\int_{\gamma} \frac{\theta}{r\sin(\theta)}\left| dz\right|.
\end{align*}
The curve $\gamma$ can be chosen such that $\theta/(r\sin(\theta))$ is bounded
along $\gamma$ when $\delta$ is chosen to be small thanks to Proposition \ref{prop:3.4}. The length of $\gamma$ can be arbitrarily small if $x_0/y$ is. We then deduce (\ref{eq:3.3}).
Moreover, this estimation is uniform over compact subsets of $V_{t_0}^+\cap [a,\infty]$. 

Hence, for any point $h_{t_0}(x)=\Phi_{t_0}(xe^{iA_{t_0}(x)})$ with $x\in V_{t_0}^+\cap [a,b]$ and  $\delta$ small enough, there exists 
$y\in V_s^+$ such that $h_{t_0}(x)/h_s(y)$ is sufficiently small. 
Let $\epsilon>0$ be a small number
and choose $b$ such that $\epsilon>\max\{ 1/h_s(x):  x\in V_{t_0+\delta}^+\cap [a,b], |s-t_0|<\delta\}$. 
Given $\epsilon>0$ and $t_0>1$, we thus have
\begin{align*}
\mathrm{supp}((\mu^{\boxtimes t_0})^{\mathrm{ac}})\cap [\epsilon, \infty)\subset&\overline{ \{1/x: x=h_{t_0}(r), r\in V_{t_0}^+\cap [a, b] \}}\\
  &\subset B_\epsilon(\mathrm{supp}(\mu^{\boxtimes s})^{\mathrm{ac}}),
\end{align*}
and
\begin{align*}
\mathrm{supp}((\mu^{\boxtimes s})^{\mathrm{ac}})\cap [\epsilon, \infty)\subset&\overline{ \{1/x: x=h_s(r), r\in V_s^+\cap [a, b] \}}\\
  &\subset B_\epsilon(\mathrm{supp}(\mu^{\boxtimes t_0})^{\mathrm{ac}}),
\end{align*}
if $\delta$ is sufficiently small. 

To discuss the behavior of $\mathrm{supp}((\mu^{\boxtimes t})^{ac})$ close to zero, we have two cases. 
If $0\in \mathrm{supp}((\mu^{\boxtimes t})^{ac})$, then $0\in \mathrm{supp}((\mu^{\boxtimes s})^{ac})$
by Proposition \ref{prop:3.5}. If $0\notin \mathrm{supp}((\mu^{\boxtimes t})^{ac})$, then $\rho$ is bounded away from $\infty$ and 
thus $ \mathrm{supp}((\mu^{\boxtimes t})^{ac})$ is bounded away zero. In both cases,
above inclusions imply
the sets $\{ \mathrm{supp}(\mu^{\boxtimes t})^{\mathrm{ac}} \}$ is Hausdorff continuous. 

Finally, atoms of $\mu^{\boxtimes t}$ change continuously as time evloves by Theorem \ref{thm:2.1} and Proposition \ref{prop:3.5}. 
This completes the proof. 
\end{proof}

\section{Estimation of norm of free multiplicative convolution semigroups}
We give an estimation of the size of the support of $\mu^{\boxtimes t}$ for a compactly supported 
probability measure on $[0, \infty)$.

\begin{proof}[Proof of Theorem \ref{thm:0.2}]
Let $a>0$ be such that $\mathrm{supp}({\rho})\subset (a,\infty)$.
As $\lim_{r\rightarrow 0}g(r)=0$, 
the set $V_t^+$ is bounded away from zero. 
Let $\alpha_t=\min\{r:  r\in V_t^+ \}>0$.
By Theorem 1.2, we have
\[
   || \mu^{\boxtimes t}||=\frac{1}{h_t(\alpha_t)}.
\]
By the choice of $\alpha_t$, we have $A_t(\alpha_t)=0$, 
\[
 g(\alpha_t)=\int_0^\infty\frac{\alpha_t(s^2+1)}{(\alpha_t-s)^2}\;d\rho(s)=\frac{1}{t-1}
\]
and 
\[
h_t(\alpha_t)=\Phi_t(\alpha_t)=\alpha_t\exp[(t-1)u(\alpha_t) ].
\]
By (\ref{eq:rep-2.1}) and the assumption $m_1(\mu)=1$, we have
\begin{align*}
 u(\alpha_t)&=a+\int_0^\infty \frac{1+\alpha_t s}{\alpha_t-s}\,d\rho(s)\\
      &=\int_0^\infty \left( \frac{1}{s}+\frac{1+\alpha_t s}{\alpha_t-s}\right)\,d\rho(s)\\
      &=\int_0^\infty\frac{\alpha_t(1+s^2)}{(\alpha_t-s)^2}\frac{\alpha_t-s}{s}\,d\rho(s)\\
      &=-\int_0^\infty\frac{\alpha_t(1+s^2)}{(\alpha_t-s)^2}\,d\rho(s)+\alpha_t \int_0^\infty\frac{\alpha_t(1+s^2)}{(\alpha_t-s)^2}\frac{1}{s}\,d\rho(s)\\ 
      &=-\frac{1}{t-1}+\alpha_t \int_0^\infty\frac{\alpha_t(1+s^2)}{(\alpha_t-s)^2}\frac{1}{s}\,d\rho(s).
\end{align*}
It is clear that $\lim_{t\rightarrow\infty} \alpha_t=0$ and hence
\[
  1=\lim_{t\rightarrow \infty} g(\alpha_t)(t-1) =\lim_{t\rightarrow \infty} (t\alpha_t)\cdot \left( \int_0^\infty\frac{s^2+1}{s^2}\,d\rho(s)\right),
\]
which yields that $\lim_{t\rightarrow \infty} (t\alpha_t)=1/V$ thanks to Lemma \ref{lemma:3.1}. 
We then obtain
\begin{align*}
\lim_{t\rightarrow\infty}u(\alpha_t)(t-1)&=-1+\lim_{t\rightarrow\infty}[(t-1)\alpha_t]\cdot \alpha_t \int_0^\infty\frac{1+s^2}{(\alpha_t-s)^2}\frac{1}{s}\,d\rho(s)\\
 &=-1
\end{align*}
as $\mathrm{supp}(\rho)$ is bounded below from zero.

Finally, we have
\begin{align*}
\lim_{t\rightarrow \infty}\frac{||\mu^{\boxtimes t} ||}{t}&=\lim_{t\rightarrow \infty}\frac{1}{t h_t(\alpha_t)}\\
   &=\lim_{t\rightarrow \infty}\left( \frac{1}{t \cdot \alpha_t}\cdot \exp[-(t-1)u(\alpha_t)] \right)\\
   &=e\cdot \int_0^\infty\frac{s^2+1}{s^2}\,d\rho(s)\\
   &=eV. 
   \end{align*}
This proves the desired result. 
\end{proof}

\begin{proposition}
Let $\mu$ be a probability measure supported on $[c, d]$ not being a Dirac measure with $c,d>0$. 
There exists $T$ such that the sets $V_t^+$ and $\mathrm{supp}(\mu^{\boxtimes t})$
have only one connected component for all $t>T$. 
\end{proposition}
\begin{proof}
The measure $\rho$ determined by  (\ref{eq:rep-2.1}) is compactly supported since $\mathrm{supp}(\mu)\subset [c,d]$. 
Hence $\lim_{r\rightarrow 0}g(r)=\lim_{r\rightarrow \infty}g(r)=0$ and 
$V_t^+$ is bounded away from zero and $\infty$. 

Let $[a,b]\subset (0,\infty)$ be a finite interval such that $\mathrm{supp}(\rho)\subset [a,b]$. Fix $t_0>1$. For any open interval $I\subset [a,b]\backslash V_{t_0}^+$,
$\rho(I)=0$ by Lemma \ref{lemma:3.2}.  Hence the function $g(r)=\int_0^\infty\frac{r(s^2+1)}{(r-s)^2}\;d\rho(s)$ has a positive local minimum at $\overline{[a,b]\backslash V_{t_0}^+}$. 
Therefore, we can choose $T$
large enough, so that $[a,b]\subset V_T^+$. The set $V_T^+$ has only one connected component as $\mathrm{supp}(\rho)\subset [a,b]$.
It then follows that $V_t^+$ and hence $\mathrm{supp}(\mu^{\boxtimes t})$
have only one connected component for all $t>T$ by Theorem \ref{density1}.
\end{proof}

\begin{proposition}\label{prop:4.2}
Let $\mu$ be a probability measure on $[c, d]$ not being a Dirac measure
for some $c,d>0$. Define
\[
a_t=\sup\{ a: a<x\quad \mathrm{for\quad  all}\quad x\in \mathrm{supp}(\mu^{\boxtimes t}) \},
\]
and
\[
b_t=\inf\{ b: b>x\quad \mathrm{for\quad all}\quad x\in \mathrm{supp}(\mu^{\boxtimes t}) \}.
\]
Then 
\[
\lim_{t\rightarrow \infty} (a_t)^{1/t}=\left(\int_0^\infty x^{-1}d\mu(x)\right)^{-1},
\quad \text{and}\quad 
\lim_{t\rightarrow \infty}(b_t)^{1/t}=\int_0^\infty x\,d\mu(x).
\]
\end{proposition}
\begin{proof}
The set $V_t^+$ is bounded away from zero and $\infty$. 
Let $\alpha_t=\min\{r:r\in V_t^+ \}>0$ and $\beta_t=\max\{r:r\in V_t^+ \}<\infty$,
we have $\lim_{t\rightarrow \infty} \alpha_t=0$ and $\lim_{t\rightarrow \infty} \beta_t=0$
by Theorem \ref{density1}. 
By adapting the argument in the proof of Theorem \ref{thm:0.2}, we have
\[
\lim_{t\rightarrow\infty}t\alpha_t={m_1(\mu)}/{V}\quad \text{and} \quad 
 \lim_{t\rightarrow\infty}(t/\beta_t)\int_0^\infty(s^2+1)\,d\rho(s)=1.
\]
On the other hand, we have
\[
\lim_{t\rightarrow\infty} \exp[u(\alpha_t)]=\lim_{z\rightarrow 0}\exp[u(z)]=\lim_{z\rightarrow 0}z/\eta_\mu(z)=1/m_1(\mu)
\]
and 
\[
\lim_{t\rightarrow\infty} \exp[u(\beta_t)]=\lim_{z\rightarrow\infty}\exp[u(z)]=\lim_{z\rightarrow \infty}z/\eta_\mu(z)
=\int_0^\infty x^{-1}d\mu(x).
\]
Therefore, we have
\begin{align*}
\lim_{t\rightarrow \infty}(a_t)^{1/t} &=\lim_{t\rightarrow\infty}(h_t(\beta_t))^{-1/t}\\
    &=\lim_{t\rightarrow\infty}\{\beta_t\exp[(t-1)u(\beta_t)]\}^{-1/t}\\
    &=\lim_{t\rightarrow\infty} \exp[-u(\beta_t)]\\
    &=\left(\int_0^\infty x^{-1}d\mu(x)\right)^{-1};
\end{align*}
and 
\begin{align*}
\lim_{t\rightarrow \infty}(b_t)^{1/t} &=\lim_{t\rightarrow\infty}(h_t(\alpha_t))^{-1/t}\\
    &=\lim_{t\rightarrow\infty}\{\alpha_t\exp[(t-1)u(\alpha_t)]\}^{-1/t}\\
    &=\lim_{t\rightarrow\infty}\exp[-u(\alpha_t)]=m_1(\mu).
\end{align*}
This finishes the proof. 
\end{proof}
\begin{remark}
Proposition \ref{prop:4.2} should be compared with the main result in \cite{HM2013}, where weak limit of
rescaled discrete semigroups 
were studied. Our method is not suitable to study weak limit, but works 
to estimate the asymptotic bound of free multiplicative convolution semigroups. 
The interested reader can easily generalize results in \cite{HM2013} to continuous free multiplicative convolution semigroups using their method.   
\end{remark}
\begin{remark}
Free multiplicative convolution semigroup can also be defined for measures on the unit circle \cite{BB2005}. 
Let $\mu$ be a probability measure on the unit circle not being a Dirac measure, it is known in \cite[Proposition 3.26]{zhong2} that 
$\mu^{\boxtimes t}$ converges to the uniform measure on the unit circle as $t\rightarrow\infty$. 
\end{remark}
\section{Acknowledgements}
The second author wants to thank Hari Bercovici, Alexandru Nica, Jiun-Chau Wang and John Williams for useful discussions. 
This work was partially supported by NSFC no. 11501423, 11431011. 
\bibliography{clipboard_sup}

\providecommand{\bysame}{\leavevmode\hbox to3em{\hrulefill}\thinspace}
\providecommand{\MR}{\relax\ifhmode\unskip\space\fi MR }
\providecommand{\MRhref}[2]{%
  \href{http://www.ams.org/mathscinet-getitem?mr=#1}{#2}
}
\providecommand{\href}[2]{#2}
\begin{thebibliography}{10}

\bibitem{exp2016}
Michael Anshelevich and Octavio Arizmendi, \emph{The exponential map in
  non-commutative probability}, Int. Math. Res. Notices \textbf{17} (2017),
  5302--5342.

\bibitem{OctavioC2012}
Octavio Arizmendi and Carlos Vargas, \emph{Products of free random variables
  and {$k$}-divisible non-crossing partitions}, Electron. Commun. Probab.
  \textbf{17} (2012), no. 11, 13. \MR{2892410}

\bibitem{BCN2012}
Serban Belinschi, Beno{\^{\i}}t Collins, and Ion Nechita, \emph{Eigenvectors
  and eigenvalues in a random subspace of a tensor product}, Invent. Math.
  \textbf{190} (2012), no.~3, 647--697. \MR{2995183}

\bibitem{BB2005}
Serban~T. Belinschi and Hari Bercovici, \emph{Partially defined semigroups
  relative to multiplicative free convolution}, Int. Math. Res. Not. (2005),
  no.~2, 65--101. \MR{2128863}

\bibitem{BB2007new}
\bysame, \emph{A new approach to subordination results in free probability}, J.
  Anal. Math. \textbf{101} (2007), 357--365. \MR{2346550}

\bibitem{MR1639647}
Hari Bercovici and Dan Voiculescu, \emph{Regularity questions for free
  convolution}, Nonselfadjoint operator algebras, operator theory, and related
  topics, Oper. Theory Adv. Appl., vol. 104, Birkh\"auser, Basel, 1998,
  pp.~37--47. \MR{1639647}

\bibitem{Biane1998}
Philippe Biane, \emph{Processes with free increments}, Math. Z. \textbf{227}
  (1998), no.~1, 143--174. \MR{1605393}

\bibitem{zhong5}
Beno{\^{\i}}t Collins, Motohisa Fukuda, and Ping Zhong, \emph{Estimates for
  compression norms and additivity violation in quantum information}, Internat.
  J. Math. \textbf{26} (2015), no.~1, 1550002, 20. \MR{3313648}

\bibitem{HM2013}
Uffe Haagerup and S\"oren M\"oller, \emph{The law of large numbers for the free
  multiplicative convolution}, Operator algebra and dynamics, Springer Proc.
  Math. Stat., vol.~58, Springer, Heidelberg, 2013, pp.~157--186. \MR{3142036}

\bibitem{HZ2014}
Hao-Wei Huang and Ping Zhong, \emph{On the supports of measures in free
  multiplicative convolution semigroups}, Math. Z. \textbf{278} (2014),
  no.~1-2, 321--345. \MR{3267581}

\bibitem{Kargin2008}
Vladislav Kargin, \emph{On asymptotic growth of the support of free
  multiplicative convolutions}, Electron. Commun. Probab. \textbf{13} (2008),
  415--421. \MR{2424965}

\bibitem{NicaS1996}
Alexandru Nica and Roland Speicher, \emph{On the multiplication of free
  {$N$}-tuples of noncommutative random variables}, Amer. J. Math. \textbf{118}
  (1996), no.~4, 799--837. \MR{1400060}

\bibitem{DVV1993}
Dan Voiculescu, \emph{The analogues of entropy and of {F}isher's information
  measure in free probability theory. {I}}, Comm. Math. Phys. \textbf{155}
  (1993), no.~1, 71--92. \MR{1228526}

\bibitem{Basic}
Dan Voiculescu, Ken Dykema, and Alexandru Nica, \emph{Free random variables},
  CRM Monograph Series, vol.~1, American Mathematical Society, Providence, RI,
  1992, A noncommutative probability approach to free products with
  applications to random matrices, operator algebras and harmonic analysis on
  free groups. \MR{1217253 (94c:46133)}

\bibitem{Lectures16}
Dan Voiculescu, Nicolai Stammeier, and Moritz Weber, \emph{Free probability and
  operator algebras}, M\"unster Lectures in Mathematics, European Mathematical
  Society, 2016.

\bibitem{John2016}
John Williams, \emph{On the hausdorff continuity of free {L}\'evy processes and
  free convolution semigroups}, J. Math. Anal. Appl. (online 2016).

\bibitem{zhong2}
Ping Zhong, \emph{Free {B}rownian motion and free convolution semigroups:
  multiplicative case}, Pacific J. Math. \textbf{269} (2014), no.~1, 219--256.
  \MR{3233917}

\bibitem{zhong6}
\bysame, \emph{On the free convolution with a free multiplicative analogue of
  the normal distribution}, J. Theoret. Probab. \textbf{28} (2015), no.~4,
  1354--1379. \MR{3422934}

\end{thebibliography}
\bibliographystyle{amsplain}

\end{document}